\documentclass[reqno,oneside,12pt]{amsart}
\usepackage{amssymb,enumerate,bm}

\usepackage[T1]{fontenc}

\usepackage[width=14cm,centering]{geometry}

\newtheorem{theorem}{Theorem}

\newtheorem{cor}[theorem]{Corollary}
\newtheorem{lemma}[theorem]{Lemma}

\theoremstyle{definition}

\newtheorem*{ack}{Acknowledgement}
\theoremstyle{remark}

\newtheorem{fact}{{\rm Fact}}

\long\def\red#1{\bgroup
 \color{red}#1\egroup}

\newif\ifcnote\cnotefalse

\cnotetrue
\title[On Polynomial Rings Over Nil Rings in Several Variables]
    {On Polynomial Rings Over Nil Rings in Several Variables and the Central Closure of Prime Nil Rings}

    \author[Chebotar, Ke, Lee and Puczy\l owski]
{M. Chebotar, W.-F. Ke, P.-H. Lee and E. R. Puczy\l owski }


\address{Department of Mathematical Sciences,
Kent State University, Kent OH 44242, U.S.A.}
\email{chebotar@math.kent.edu}

\address{Department of Mathematics and Research Center for Theoretical Sciences, National Cheng
Kung University, Tainan 701, Taiwan} \email{wfke@mail.ncku.edu.tw}

\address{Department of Mathematics, National Taiwan University,
and National Center for Theoretical Sciences, Taipei 106, Taiwan} \email{phlee@math.ntu.edu.tw}

\address{Institute of Mathematics, University of Warsaw, Warsaw,
Poland} \email{edmundp@mimuw.edu.pl}
\begin{document}
\begin{abstract}
We prove that the ring of polynomials in several commuting indeterminates over a nil ring cannot be homomorphically mapped onto a ring with identity, i.e. it is Brown-McCoy radical. It answers a question posed by Puczy{\l}owski and Smoktunowicz.
We also show that the central closure of a prime nil ring cannot be a simple ring with identity solving a problem due to Beidar.
\end{abstract}
\maketitle

\section{Introduction}

Let $R[x]$ be a polynomial ring over a nil ring $R$.
Many recent papers on polynomial rings in one indeterminate over nil rings were inspired by the following three questions:

\begin{enumerate}
\item Is it true that $R[x]$ is nil?
\item Is it true that $R[x]$ is Jacobson radical?
\item Is it true that $R[x]$ is Brown-McCoy radical?
\end{enumerate}

The first question was posed by Amitsur \cite{A71} and Krempa \cite{K72}, and was answered negatively by Smoktunowicz \cite{S00}. The second question is open, and is equivalent to the famous Koethe's Conjecture as shown by Krempa in \cite{K72}. The third question was posed by Puczy\l owski \cite{P91}, and was answered positively by Puczy\l owski  and Smoktunowicz \cite{PS98}. The results answering questions (1) and (3) can be considered as approaches towards a solution of the Koethe Conjecture.

It is not surprising that the question whether the ring of polynomials in two or more commuting indeterminates over a nil ring is Brown-McCoy radical attracted a lot of attention. Originally posed in the paper by Puczy\l owski  and Smoktunowicz \cite[Question 1a]{PS98}, this question appears in a number of papers and surveys dedicated to the Koethe Conjecture and related topics such as \cite[p.~210]{FW03}, \cite[Question 2.12a]{P06}, \cite[p.140]{PW07}, \cite[Question 26]{S01} and \cite[p.~235]{S03}, to mention just a few. The problem again had a clear connection to Koethe's Conjecture, as the theorem due to Smoktunowicz \cite{S03} showed that the existence of a nil ring $R$ such that $R[x,y]$ is not Brown-McCoy radical would give a counterexample to the Koethe Conjecture. In 2008 the authors published
the paper \cite{CKLP08} where they proved that a polynomial ring $R[x,y]$ is Brown-McCoy radical provided that $R$ is a nil algebra over a field of nonzero characteristic. However, it is well-known that questions (1)--(3) may have different answers for nil algebras over fields of different nature. For example, it was not clear whether the result of \cite{CKLP08} could be extended to nil algebras over the field of rational numbers. One more doubt came from another striking result by Smoktunowicz \cite[Theorem 1.3]{S09}: Over every countable field $K$ there is a nil algebra $R$ such that the polynomial ring
$R[x_1,\ldots,x_6]$ in six commuting indeterminates $x_1,\ldots,x_6$ over $R$ contains a noncommutative
free $K$-algebra of rank two. The existence of such a surprising ring makes one imagine that it would be quite possible that this polynomial ring or a similar ring can be mapped onto a ring with $1$.

Motivated by the above question, we study and present the following result on prime nil rings.

\begin{theorem}\label{Main}
Let $R$ be a prime nil ring with extended centroid $C$. Let $c_{1},\ldots,c_{n}\in C$, $n\ge 1$, be such that $S=R[c_{1},\ldots,c_{n}]$ is a simple ring. Then $S$ has a zero center.
\end{theorem}
The extended centroid $C$ of a prime ring $R$ can be defined as the center of the maximal right ring of quotients of $R$,
and the ring $RC$ is called the central closure \cite[p.~68]{BMM96}. The central closure approach to handle the above mentioned problem was advocated by Ferrero and Wisbauer \cite{FW03} and Beidar \cite{PW07}, but again there were some serious doubts whether such approach could work. One obstacle was that there were no known solutions of question (3) that relied on the central closure approach. Another and more serious obstacle was connected to the following question by Beidar \cite[p.~140]{PW07}: Does there exist a prime ring with zero center whose central closure is a simple ring with identity? It is natural to expect that such ring does not exist, but the example constructed by Chebotar \cite{C08} showed that the situation is more complicated. Fortunately, together with some ideas from Convex Geometry, the central closure approach does work.

We conclude this section with some applications of Theorem~\ref{Main}.

First of all, the following theorem follows immediately from Theorem \ref{Main} and \cite[Theorem~4.12]{FW03}, and answers a question posed by Puczy\l{}owski and Smoktunowicz \cite[Question~1a]{PS98}.

\begin{theorem}\label{PS}
Let $R$ be a nil ring. Then the ring of polynomials in two or more commuting indeterminates over $R$ is Brown-McCoy radical.
\end{theorem}

\begin{proof}
Assume on the contrary that $R[x_{1},\dots,x_{n}]$ is not Brown-McCoy radical for some set $\{x_{1},\dots,x_{n}\}$ of commuting indeterminates over $R$. Then there exists a maximal ideal $M$ of $R[x_{1},\dots,x_{n}]$ such that $R[x_{1},\dots,x_{n}]/M$ is a simple ring with identity. Replacing $R$ by $R/(M\cap R)$, we may assume that $R$ is prime and $M\cap R=0$. By \cite[Theorem~4.12]{FW03}, there exist $c_{1},\dots,c_{m}\in C$, the extended centroid of $R$, such that $RC=R[c_{1},\dots,c_{m}]$ is simple with identity, a contradiction by Theorem~\ref{Main}.
\end{proof}

As a corollary, the next result answers a question \cite[Question 2.13b]{P06} raised by Beidar.

\begin{cor}
The central closure of a prime nil ring contains no identity.
\end{cor}
\begin{proof}
Let $A$ be a prime nil ring. Assume that $a_{1}c_{1}+\dots+a_{n}c_{n}=1$ where $a_{i}\in A$ and $c_{i}\in C,$ the extended centroid of $A$.
Let $R$ be a subring of $A$ generated by $a_1,\ldots,a_n$.
The map $\theta:R[x_{1},\dots,x_{n}]\to R[c_{1},\dots,c_{n}]$ given by $\theta(p(x_{1},\dots,x_{n}))=p(c_{1},\dots,c_{n})$ is a ring
epimorphism, and $\theta(c_{1}x_{1}+\dots+c_{n}x_{n})=a_{1}c_{1}+\dots+a_{n}c_{n}=1$, which is not possible by Theorem \ref{PS}.
\end{proof}
The next result answers an old question posted by Puczy\l{}owski \cite[Question 13b]{P91}.  It can be deduced from
Theorem \ref{PS} by means of \cite[Theorem~5.1]{FW03} or \cite[Remark~1]{S03}.
\begin{theorem}\label{T4}
Let $R$ be a nil ring. Then the ring of polynomials in two or more noncommuting indeterminates over $R$ is Brown-McCoy radical.
\end{theorem}

Finally, we make a remark that the above results can help to answer some other open problems in Radical Theory. For example, using Theorem \ref{PS} together with \cite[Proposition 5.5]{FW03}, one sees that for any ring $R$ the upper nil radical is contained in the $u$-strongly prime radical of $R$, a question formulated by Kau\v cikas and Wisbauer \cite[Problem]{KW}. Also, as noted in \cite[p.~7]{P06}, Theorem~\ref{T4} implies that for two nil algebras $A$ and $B$, the algebra $A\otimes B$ is Brown-McCoy radical. This answers Question 23 in \cite{P91}.

\section{Simple Facts from Convex Geometry}

To make the paper a bit self-contained we record several facts about convex polytopes which will be used in the proof of the Main Theorem. For a concise reading about basic properties of convex polytopes, one is referred to \cite[Chapter~16]{GO04}.

A convex polytope in the $n$-dimensional real space $\mathbb{R}^n$ is the convex hull of a finite set. Here, the convex hull of the set $X=\{x_1,\dots,x_m\}\subseteq\mathbb R^n$ is defined to be
$$
\textup{conv}(X) = \biggl\{\sum_{i=1}^m \lambda_ix_i\biggm| \lambda_i\geq 0, \sum_{i=1}^m \lambda_i=1\biggr\}.
$$
Equivalently, a convex polytope is a bounded solution set of a finite system of linear inequalities:
$$
P(A,b)=\bigl\{x\in\mathbb R^n \bigm| a_i^Tx\leq b_i\text{ for all }1\leq i\leq m\bigr\},
$$
where $A$ is a real $m\times n$ matrix with rows $a_i^T$, and $b\in\mathbb R^m$ with entries $b_i$. Here boundedness means that there is a constant $N$ such that
$\|x\| < N$ holds for all $x \in P(A,b)$.

A Minkowski sum of two sets of $A$ and $B$ in  $\mathbb{R}^n$ is formed by adding each vector in $A$ to each vector in $B$.

\begin{fact} The intersection of finitely many convex polytopes is a convex polytope. And, a cut-off of a convex polytope by a hyperplane is also a convex polytope.
\end{fact}

This follows from the definition of the convex polytope.

\begin{fact}
  If $S$ is a convex set then  $\mu S+\lambda S$ (in the Minkowski sum sense) is also a convex set for all positive integers $\mu$ and $\lambda$. Moreover, $\mu S+\lambda S=(\mu +\lambda)S$.
\end{fact}

This follows from the definitions of the convex set and the Minkowski sum.

\begin{fact}
  Let $B_2 \subset B_1$ be two balls with the same center. Then there exists a convex polytope $D$ such that $B_2\subseteq D \subseteq B_1$.
\end{fact}

This fact can be deduced, for example, from \cite[Theorem 1.8.19]{Schneider}  which asserts that a convex body can be approximated arbitrarily closely by convex polytopes.

The following is the key fact to be used in the proof of Theorem~\ref{Main}. The idea of the proof was suggested to us by Fedor Nazarov.

\begin{theorem}\label{N}
Let $P$ be  a convex polytope. Then there is an integer
$\mu$  and convex polytopes $K_1, \ldots, K_{\mu}$ so that:
\begin{enumerate}
\item $P$ is a union of $K_1, \ldots, K_{\mu}$.
\item Each of $K_i$ can be bounded by a ball of radius $1$.
\item For each positive integer $\nu$, the union $\cup_{i=\nu}^{\mu} K_i$ is a convex polytope.
\end{enumerate}
\end{theorem}

\begin{proof}
Let $B_1$ be the smallest bounding ball for $P$. Let $R$ be the radius of $B_1$ and $O$ its center. If $R=1$, we have nothing to prove, so assume that $R>1$. We will choose a ball $B_3$ with the same center $O$ which is sufficiently close to $B_1$. More precisely, if $D$ is a convex polytope such that $B_3\subseteq D \subseteq B_1$ then the cut-off of $B_1$ (the part not containing $B_3$) by any of its facet-defining hyperplanes can be bounded by a ball of radius $1$.
 This can be done in the following way. First, pick a point $A$ on $B_1$ and make a ball $B_0$ of radius $\frac12$ centered at $A$. The
intersection of the surfaces of $B_1$ and $B_0$, which are $(n-1)$-spheres, will be an $(n-2)$-sphere and is contained in some hyperplane $H$.
We then pick the ball $B_3$ centered at $O$ having $H$ as a tangent hyperplane. Write the radius of $B_3$ as $(1-2\gamma)R$ for some $\gamma>0$.

Now let $B_2$ be the ball centered at $O$ of radius $(1-\gamma)R$.  Let $D_1$ be a convex polytope such that $B_3\subseteq D_1 \subseteq B_2$. Order the facets of $D_1$ as $F_{1,1},\dots,F_{1,r_1}$.
For each facet $F_{1,i}$, the facet-defining hyperplane $H_{1,i}$ cuts off a convex polytope $K_{1,i}$ from $P$ which does not contain~$B_3$ in a sequel: First, $H_{1,1}$ cuts off the piece $K_{1,1}$ from $P$, then $H_{1,2}$ cuts off the piece $K_{1,2}$ from the remaining, and so on. Let $P_1=P\cap D_1$ which is again a convex polytope.
By our construction, we have that each $K_{1,i}$ is bounded by a ball of radius $1$, $P=(\cup_{i=1}^{r_1} K_{1,i}) \cup P_1$, and for any integer $t$, the union
$(\cup_{i=t}^{r_1} K_{1,i}) \cup P_1$ is a convex polytope.

Note that $P_1$ is bounded by the ball $B_2$ of radius $(1-\gamma)R<R$.  If $(1-\gamma)R<1$, then we are done. Assume this is not the case, and let $B_4$ be the ball centered at $O$ of radius $(1-3\gamma)R$, and let $D_2$ be a convex polytope such that $B_4\subseteq D_2\subseteq B_3$.
Order the facets of $D_2$ as $F_{2,1},\dots,F_{2,r_2}$.
For each facet $F_{2,i}$, the facet-defining hyperplane $H_{2,i}$ cuts off a convex polytope $K_{2,i}$ from $P_1$ which does not contain~$B_4$ in a sequel: First, $H_{2,1}$ cuts off the piece $K_{2,1}$ from $P_1$, then $H_{2,2}$ cuts off the piece $K_{2,2}$ from the remaining, and so on. Let $P_2=P_1\cap D_2$ which is again a convex polytope.
By our construction, we have that each $K_{2,i}$ is bounded by a ball of radius $1$, $P_1=(\cup_{i=1}^{r_2} K_{2,i}) \cup P_2$, and for any integer $t$, the union
$(\cup_{i=t}^{r_2} K_{2,i}) \cup P_2$ is a convex polytope.

We continue in this fashion until we have $(1-s\gamma)R<1$ for some $s\geq0$. Then we rename $K_{1,1}$ as $K_1$, $K_{1,2}$ as $K_2$, $\dots$, $K_{s,r_s}$ as $K_{\mu-1}$, and $P_s$ as $K_\mu$. These $K_i$, $i=1,\dots,\mu$, will fulfill the required conditions.
 \end{proof}

\section{Proof of the Main Theorem}
Let $R$ be a prime ring with extended centroid $C$ and let $c_1,\ldots,c_n$ be elements of $C$.
For any multi-index $\alpha = (\alpha_1,\ldots, \alpha_n)$, where each $\alpha_i$ is a non-negative integer, let
$$c^{\alpha }=\prod _{i=1}^{n}c_{i}^{\alpha _{i}}=c_{1}^{\alpha _{1}}\cdots c_{n}^{\alpha _{n}}.$$
Like in case of polynomial rings in several variables, let us call $c^{\alpha}$ a monomial element and
define its (total) degree $|\alpha|$ as $|\alpha |=\sum _{i=1}^{n}\alpha _{i}$.
An element $p\in R[c_1,\ldots,c_n]$ can be represented as a polynomial expression
\begin{equation}\label{E1}
 p=\sum _{\alpha }a_{\alpha }c^{\alpha },
\end{equation}
where $a_{\alpha }=a_{\alpha _{1},\ldots ,\alpha _{n}}\in {R}$.

Slightly modifying ideas of \cite[p.\ 2474]{PS98} we define the degree $\deg(p)$ of a nonzero element $p$ represented by formula (\ref{E1}) as the largest degree of a monomial occurring with nonzero coefficient in the expression of $p$, the minimal degree $\min(p)$ as the smallest degree of a monomial occurring with nonzero coefficient and the length $l(p)=\deg(p)-\min(p)+1$.

Let us note that the same element may have many different polynomial representations. For an ideal $I$ of $R$, we write $p\in I[c_1,\ldots,c_n]$ to mean that we are looking at the polynomial representation of $p$ with coefficients in the ideal $I$ of $R$.

Our first lemma is inspired by \cite[Lemma 1]{PS98}. We basically use the same argument presented there with just a slight modification.

\begin{lemma}\label{L1}
Let $R$ be a prime ring with extended centroid $C$, and let $c_{1},\ldots,c_{n}\in C$, $n\ge 1$. Suppose that
$S=R[c_{1},\ldots,c_{n}]$ is a simple ring with $1$. Let $w\in  R[c_1,\ldots,c_n]$ be a polynomial representation of $1$
such that $\min(w)\ge 1$. Then for any nonzero ideal $I$ of $R$ there exists a polynomial representation
$t\in  I[c_1,\ldots,c_n]$ of $1$ with $\min(t)\ge \deg(w)$ and $l(t)\le \deg(w)$.
\end{lemma}
\begin{proof}
Denote $d=\deg(w)$. Let $J$ be a nonzero ideal of $R$ such that $c_1^{-d}J\subseteq R$. Since $S$ is a simple ring and $J[c_{1},\ldots,c_{n}]$ a nonzero ideal of $S$, we have $S=J[c_{1},\ldots,c_{n}]$.
Let $p=p_k+p_{k+1}+\ldots+p_l \in J[c_1,\ldots,c_n]$ be a polynomial representation of $1$ where $p_i$ is a sum of monomials of degree $i$, with $0\le k\le i \le l$.  Using the definition of $J$, every coefficient $j\in J$ in the expression of $p$ can be replaced
by $c_1^dr$, for some $r\in R$, so there exists a polynomial representation $q=q_{\eta}+\ldots+q_{\zeta} \in R[c_1,\ldots,c_n]$ of $1$  where
$q_i$ is a sum of monomials of degree $i$, with $d\le \eta\le i \le \zeta$. For any polynomial representation of $1$ with
coefficients in a nonzero ideal $I$ we can multiply it by $q$, so we can assume that $s=s_{\gamma}+\ldots+s_{\delta} \in I[c_1,\ldots,c_n]$ is a polynomial representation of~$1$  where
$s_i$ is a sum of monomials of degree $i$, with $d\le \gamma\le i \le \delta$.
Assume that $s_{\gamma}$ is not zero. Let $s^*=w s_\gamma+s_{\gamma+1}+\ldots+s_{\delta}$.
We get that $s^* \in  I[c_1,\ldots,c_n]$ is a polynomial representation of~$1$. Since $\min(w)\ge 1$, we have
$\min (s^*)\ge \min (s)+\nobreak 1$.
If $\deg (ws_{\gamma})>\delta$,
then $\deg(w)+\deg(s_\gamma)=\deg(w)+\gamma>\delta$, and so $l(s)=\delta-\gamma+1\le \deg(w)$. In this case, we are done as we can use $s$ as the required polynomial representation of $1$.
So assume that $\deg (ws_{\gamma})\le \delta$,  and then $l(s^*)\le l(s)-1$.   Repeating the $*$-operation enough
many times we arrive at the polynomial representation of $1$ denoted by  $t\in  I[c_1,\ldots,c_n]$ with $\min(t)\ge \deg(w)$ and $l(t)\le \deg(w)$.
\end{proof}

\begin{lemma}\label{L2}
Let $R$ be a prime ring with extended centroid $C$, and let $c_{1},\ldots,c_{n}\in C$, $n\ge 1$. Suppose that
$S=R[c_{1},\ldots,c_{n}]$ is a simple ring with $1$. Then there exist a positive integer $m$ and $a_{i_1,\ldots,i_n}\in R$,
with $i_j\ge 0$, $i_1+\ldots+i_n=m$,
such that
$$
\sum_{i_1+\ldots+i_n=m} a_{i_1,\ldots,i_n} c_1^{i_1}\cdots c_n^{i_n}=1.
$$
\end{lemma}
\begin{proof}
Let $I$ be a nonzero ideal of $R$ such that $c_1^{-1}I \subseteq R$. Since $S$ is a simple ring and $I[c_{1},\ldots,c_{n}]$ a nonzero ideal of $S$, we have $S=I[c_{1},\ldots,c_{n}]$. Thus, there is a representation of $1$ of the form (\ref{E1}) with coefficients $a_{\alpha}\in I$. Replacing
$a_\alpha$ by $c_1 b_\alpha$ where $b_\alpha \in R$ we get that $1$ can be written as a polynomial expression
$w \in R[c_1,\ldots, c_n]$ with $\min(w)\ge 1$.

Denote $d=\deg(w)$. Let $L$ be a nonzero ideal of $R$ such that $c_1^{-j} L\subseteq R$ for all $j=1,\ldots,d$. According to Lemma~\ref{L1} there exists a polynomial representation $t\in  L[c_1,\ldots,c_n]$ of $1$ with $\min(t)\ge d$ and $l(t)\le d$.

Now, if $a\in L$ and $0\leq j\leq d$, then for some suitable $b_j\in R$, we have $a=b_jc_1^j$.
For a monomial $ac_1^{\alpha_1}c_2^{\alpha_2}\cdots c_n^{\alpha_n}$ of $t$, where $a\in L$, we have $\min(t)\leq |\alpha|\leq\deg(t)$. If $j=\deg(t)-|\alpha|$, then $j< l(t)\le d$, and we can rewrite $ac_1^{\alpha_1}c_2^{\alpha_2}\cdots c_n^{\alpha_n}$ as
$bc_1^{\beta}c_2^{\alpha_2}\cdots c_n^{\alpha_n}$, where $b\in R$ and $\beta=\alpha_1+j$. Therefore, $\beta+\sum_{i=2}^n\alpha_i=\deg(t)$.

This way, we rewrite every monomial of $t$ into a monomial with coefficients in $R$, and degree $\deg(t)$. The resulting polynomial $g$ is an element in $R[c_1,\ldots,c_n]$ with the same $\min(g)$ and $\deg(g)$, i.e., $l(g)=1$, which proves the lemma.
\end{proof}

Replacing the ring $R$ in Lemma~\ref{L2} by any nonzero ideal $I$ of $R$, we immediately get

\begin{cor}\label{C1}
Let $R$ be a prime ring with extended centroid $C$, and let $c_{1},\ldots,c_{n}\in C$, $n\ge 1$. Suppose that
$S=R[c_{1},\ldots,c_{n}]$ is a simple ring with $1$ and $I$ a nonzero ideal of $R$.
Then there exist a positive integer $m$ and $a_{i_1,\ldots,i_n}\in I$,
with $i_j\ge 0$, $i_1+\ldots+i_n=m$,
such that
$$
\sum_{i_1+\ldots+i_n=m} a_{i_1,\ldots,i_n} c_1^{i_1}\cdots c_n^{i_n}=1.
$$
\end{cor}

{\bf Proof of Theorem~\ref{Main}.}
Assume that the simple ring $S$ has a nonzero center, so that $S$ contains an identity element, and we will derive a contradiction.

Let $I$ be a nonzero ideal of $R$ such that $c_ic_j^{-1} I \subseteq R$ for all $i,j=1,\ldots, n$. Observe that $I$ is chosen in such a way that
\begin{align*}
  \hbox{for any } a\in I, ac_i = bc_j \hbox{ for some }b\in R.
\end{align*}
Thus, if $a\in I$, then for some $b\in R$
\begin{align}\label{eq:change}
a c_1^{\alpha_1}\cdots c_i^{\alpha_i+1}\cdots c_j^{\alpha_j-1}\cdots c_n^{\alpha_n}=bc_1^{\alpha_1}\cdots c_i^{\alpha_i}\cdots c_j^{\alpha_j}\cdots c_n^{\alpha_n}.
\end{align}

Let $J_0=I^{2n}$ and $L_0=J_0 I$. These two parts of $L_0$ (and those of $L_i$ later) will play different but important roles in the proof.

By Corollary~\ref{C1} for the nonzero ideal $L_0$, there exist a positive integer $m$ and $a_{i_1,\ldots,i_n}\in L_0$, with $i_j\ge 0$, $i_1+\ldots+i_n=m$,
such that
$$
p=\sum_{i_1+\ldots+i_n=m} a_{i_1,\ldots,i_n} c_1^{i_1}\cdots c_n^{i_n}=1.
$$

Let $P$ be a convex hull of the set
$$\{(i_1,\ldots,i_n)\mid i_j\geq0, 1\leq j\leq n, \hbox{ and }i_1+\dots+i_n=m\}.$$
By Theorem~\ref{N}, $P$ is a union of convex polytopes $K_1, \ldots, K_{\mu}$ such that
each of $K_i$ can be bounded by a ball of radius $1$, and for each positive integer $\nu$, the union $\cup_{i=\nu}^{\mu} K_i$ is a convex polytope.

We point out that $K_i$ can be  bounded by a ball of radius $1$ means that the (Euclidean) distance between two exponential vectors $(i_1,\ldots,i_n)$ and $(j_1,\ldots,j_n)$ (of $c_1^{i_1}\cdots c_n^{i_n}$ and $c_1^{j_1}\cdots c_n^{j_n}$ respectively) which belong to the same $K_i$ is at most~$2$. For any positive integer $\ell$, any two exponential vectors $(i_1,\ldots,i_n)$ and $(j_1,\ldots,j_n)$ in $\ell K_i$ would have distance at most $2\ell$, and we have
$$|i_1-j_1|+\dots+|i_n-j_n|\le \sqrt{n}\left[(i_1-j_1)^2+\dots+(i_n-j_n)^2\right]^{\frac{1}{2}}\le \sqrt{n} (2\ell)\le 2n\ell.$$
Since $i_1+\dots+i_n=j_1+\dots+j_n$, using (\ref{eq:change}), we see that if $a=a_1a_2\cdots a_{2n\ell}$, where each $a_i\in I$, then there is some $b\in R$ such that $ac_1^{i_1}\cdots c_n^{i_n}=bc_1^{j_1}\cdots c_n^{j_n}$. Consequently, we have the following statement.
\begin{equation}\label{eq:change-general}
\lower 1pc\hbox to .8\linewidth{\vbox{\rightskip=.2\linewidth\noindent
\textit{For $(i_1,\ldots,i_n),(j_1,\ldots,j_n)\in\ell K_i$, $\ell\geq1$, $\epsilon\geq0$, and $a\in I^{2n\ell+\epsilon}$, there is some $b\in R$ such that $ac_1^{i_1}\cdots c_n^{i_n}=bc_1^{j_1}\cdots c_n^{j_n}$.}}}
\end{equation}

Let $w=\sum_{\alpha\in K_1}a_\alpha c^\alpha$ be the sum of the monomials of $p$ having the exponential vectors in $K_1$. We claim that $w$ is nilpotent. Pick an arbitrary monomial $a_{\alpha_0} c^{\alpha_0}$ from $w$, $\alpha_0\in K_1$. We have $a_{\alpha}\in L_0=J_0I=I^{2n+1}$ for all $\alpha\in K_1$. If $\alpha'\in K_1$, then there is some $b'\in R$, depending on $\alpha'$, such that $a_{\alpha'}c^{\alpha'}=b'c^{\alpha_0}$ as we have seen in (\ref{eq:change-general}). In this way, we see that $w=(a_{\alpha_0}+\sum_{\alpha'\in K_1\setminus\{\alpha_0\}}b')c^{\alpha_0}$. Since $a_{\alpha_0}+\sum_{\alpha'\in K_1\setminus\{\alpha_0\}}b'\in R$ is nilpotent, we conclude that $w=\sum_{\alpha\in K_1}a_\alpha c^\alpha$ is indeed nilpotent.

Now we write $p=w+q$, where the exponential vectors of the monomials of $q$ all fall into $P\setminus K_1\subseteq P_1=\cup_{i=2}^{\mu} K_i$, and $P_1$ is a convex polytope.

Due to the Minkowski sum property, for any positive integer $\ell$ the monomials of $q^\ell$ will have the following property: the convex hull of their exponential vectors is contained in $\ell P_1$.

If $w$ is zero, we would simply proceed to the next iteration. Thus, without loss of generality, we assume that $w$ is nonzero. Let $k_1$ be a positive integer such that $w^{k_1}=0$. Let $l_1=mk_1$, and set
$$
p_1=p^{l_1}=(q+w)^{l_1}=
q^{l_1}+\binom{l_1}{1}q^{l_1-1}w+\ldots+\binom{l_1}{k_1-1}q^{l_1-k_1+1}w^{k_1-1}.
$$
All monomials in $p_1$ have coefficients in $L_0^{l_1}= J_0^{l_1}I^{l_1}$. We put $J_1=J_0^{l_1}$, and $L_1=J_1I$.

Using (\ref{eq:change-general}), we will now rewrite the monomials of $p_1$ so that the resulting monomials will have their exponential vectors lying in the convex polytope $l_1P_1$ and the coefficients in~$L_1$.

For $0\le s\leq l_1$ and $0\le t< k_1$, $s+t=l_1$, we look at $q^sw^t$. Let $a_\alpha b_\beta  c^{\alpha}c^{\beta}$ be a monomial occurring in $q^sw^t$, where $a_\alpha c^\alpha$ is a monomial in $w^t$, and $b_\beta c^\beta$ is a monomial in $q^s$. Note that $a_\alpha b_\beta\in L_0^{l_1}=J_0^{l_1}I^{l_1}$, and so $a_\alpha b_\beta=u_0v_0$ for some $u_0\in J_0^{l_1}$ and $v_0\in I^{l_1}$. (In general, $J_0^{l_1} I^{l_1}$ consists of the finite sums of products of elements from $J_0^{l_1}$ and that from $I^{l_1}$, but here we are only looking at a single monomial $a_\alpha b_\beta  c^{\alpha}c^{\beta}$.)

Choose an arbitrary monomial $y_0=b_\gamma c^\gamma$ in $q$ (thus, $\gamma\in P_1$). As $c^{\alpha}$ and $y_0^t$ have the same degree $tm<k_1m$, there is some element $v_0'\in RI^{l_1-tm}$ such that $a_\alpha b_\beta c^\alpha=u_0v_0c^\alpha=u_0v_0' c^{t\gamma}$. Thus $a_\alpha b_\beta  c^{\alpha}c^{\beta}=u_0v_0'c^{t\gamma+\beta}$. From $t\gamma+\beta\in tP_1+sP_1= l_1P_1$, we see that $u_0v_0'c^{t\gamma+\beta}$ is a monomial with the exponential vector in $l_1P_1$ and coefficient in $J_0^{l_1}I$.

As a result, we get
$$
p_1=\sum_{i_1+\ldots+i_n=ml_1} b_{i_1,\ldots,i_n} c_1^{i_1}\cdots c_n^{i_n}=1,
$$
where all coefficients $b_{i_1,\ldots,i_n}$ of $p_1$ are in $L_1$ and all exponential vectors of the monomials fall into $l_1P_1=\cup_{i=2} l_1K_i$.

To illustrate the iteration process,
we let $w_1=\sum_{\alpha\in l_1K_2}b_\alpha c^\alpha$ be the sum of the monomials of $p_1$ having their exponential vectors in $l_1K_2$. Again, we claim that $w_1$ is nilpotent. Pick an arbitrary monomial $b_{\alpha_1} c^{\alpha_1}$ from $w_1$, $\alpha_1\in l_1K_2$.  For all $\alpha\in l_1K_2$, we have $b_{\alpha}\in L_1=J_1I=I^{2l_1n+1}$. If $\alpha''\in l_1K_2$, then there is some $b''\in R$, depending on $\alpha''$, such that $b_{\alpha''}c^{\alpha''}=b''c^{\alpha_1}$ as we have seen in (\ref{eq:change-general}). In this way, we see that $w_1=(b_{\alpha_1}+\sum_{\alpha''\in l_1K_2\setminus\{\alpha_1\}}b'')c^{\alpha_1}$. Since $b_{\alpha_1}+\sum_{\alpha''\in l_1K_2\setminus\{\alpha_1\}}b''\in R$ is nilpotent, we conclude that $w_1=\sum_{\alpha\in l_1K_2}b_{\alpha} c^{\alpha}$ is indeed nilpotent.

Next we write $p_1=w_1+q_1$, where the exponential vectors of the monomials of $q_1$ all fall into $l_1P_1\setminus l_1 K_2\subseteq P_2=\cup_{i=3}^{\mu} l_1K_i$, and $P_2$ is a convex polytope.

Again, due to the Minkowski sum property, for any positive integer $\ell$ the monomials of $q_1^\ell$ will have the property that the convex hull of their exponential vectors is contained in $\ell P_2$.

If $w_1$ is zero, we would proceed to the next iteration. Thus, we assume that $w_1$ is nonzero. Let $k_2$ be a positive integer such that $w_1^{k_2}=0$. Let $l_2=l_1mk_2$, and set
$$
p_2=p_1^{l_2}=(q_1+w_1)^{l_2}=
q_1^{l_2}+\binom{l_2}{1}q_1^{l_2-1}w_1+\ldots+\binom{l_2}{k_2-1}q_1^{l_2-k_2+1}w_1^{k_2-1}.
$$
All monomials in $p_2$ have coefficients in $L_1^{l_2}= J_1^{l_2}I^{l_2}$. We put $J_2=J_1^{l_2}$, and $L_2=J_2I$.

Using (\ref{eq:change-general}), we will rewrite the monomials of $p_2$ so that the resulting monomials will have their exponential vectors lying in the convex polytope $l_2P_2$ and the coefficients in~$L_2$.

For $0\le s\leq l_2$ and $0\le t< k_2$, $s+t=l_2$, we look at $q_1^sw_1^t$. Let $b_{\alpha''} e_{\beta''}  c^{\alpha''}c^{\beta''}$ be a monomial occurs in $q_1^sw_1^t$, where $b_{\alpha''}c^{\alpha''}$ is a monomial in $w_1^t$, and $e_{\beta''} c^{\beta''}$ is a monomial in $q_1^s$. Note that $b_{\alpha''} e_{\beta''}\in L_1^{l_2}=J_1^{l_2}I^{l_2}$, and so $b_{\alpha''} e_{\beta''}=u_1v_1$ for some $u_1\in J_1^{l_2}$ and $v_1\in I^{l_2}$.

Choose an arbitrary monomial $y_1=b_{\gamma''} c^{\gamma''}$ in $q_1$ (thus, $\gamma''\in P_2$). As $c^{\alpha''}$ and $y_1^t$ have the same degree $l_1tm<l_1k_2m$, there is some element $v_1''\in RI^{l_1k_2m-l_1tm}$ such that $b_{\alpha''} e_{\beta''} c^{\alpha''}=u_1v_1''c^{\alpha''}=u_1v_1'' c^{t\gamma''}$. Thus $b_{\alpha''} e_{\beta''} c^{\alpha''}c^{\beta''}=u_1v_1''c^{t\gamma''+\beta''}$. From $t\gamma''+\beta''\in tP_2+sP_2= l_2P_2$, we see that $u_1v_1''c^{t\gamma''+\beta''}$ is a monomial with the exponential vector in $l_2P_2$ and coefficient in $J_1^{l_2}I$.

As a result, we get
$$
p_2=\sum_{i_1+\ldots+i_n=ml_1l_2} e_{i_1,\ldots,i_n} c_1^{i_1}\cdots c_n^{i_n}=1,
$$
where all coefficients $e_{i_1,\ldots,i_n}$ of $p_2$ are in $L_2$ and all exponential vectors of the monomials fall into $l_2P_2=\cup_{i=3} l_1l_2K_i$.

As we proceed the iteration steps, we will finally arrive at
$$
p_{\mu-1}=\sum_{i_1+\ldots+i_n=ml_1l_2\cdots l_{\mu-1}} z_{i_1,\ldots,i_n} c_1^{i_1}\cdots c_n^{i_n}=1,
$$
where all coefficients $z_{i_1,\ldots,i_n}$ of $p_{\mu-1}$ are in $L_{\mu-1}=J_{\mu-1}I$, $J_{\mu-1}=J_0^{l_1l_2\cdots l_{\mu-1}}$, and all exponential vectors of the monomials of $p_{\mu-1}$ fall into $l_{\mu-1}P_{\mu-1}= l_1l_2\cdots l_{\mu-1}K_{\mu}$. Again, using $J_{\mu-1}$, we can write $p_{\mu-1}$ as a single monomial with coefficient in $R$, and so $p_{\mu-1}$ is nilpotent. But then, $1=p_{\mu-1}$ is nilpotent, a contradiction.

\begin{ack}
We would like to thank the National Center of Theoretical Sciences, Taiwan, for the support in organizing short visits of the first and the fourth authors.
\end{ack}

\end{document}